\documentclass[psamsfonts]{amsart}

 \makeatletter
       \def\@makefnmark{%
               \leavevmode
               \raise.9ex\hbox{\check@mathfonts
                       \fontsize\sf@size\z@\normalfont%
                               \@thefnmark}%
       }
       \makeatother
       
\usepackage{amsmath,amssymb,mathrsfs}
\usepackage[dvipdfmx,usenames]{color}
\usepackage[all]{xy}
\usepackage[dvipdfmx]{graphicx}
\usepackage{wrapfig}
\usepackage{tabularx}
\usepackage{supertabular}
\usepackage{setspace}
\usepackage{xypic}
\usepackage{enumerate} 

\newtheorem{dfn}{Definition}[section]
\newtheorem{thm}[dfn]{Theorem}
\newtheorem{lem}[dfn]{Lemma}
\newtheorem{prp}[dfn]{Proposition}
\newtheorem{rmk}[dfn]{Remark}

\newtheorem{mthm}{Theorem}

\newcommand{\ind}{\mathrm{index}}

\title[Cyclic cocycle and relative-partitioned index theorem]
{A cyclic cocycle and relative index theorems on partitioned manifolds}
\author{Tatsuki SETO}

\address{Graduate School of Mathematics, Nagoya University, Furocho, Chikusaku, Nagoya, Japan}
\email{m11034y@math.nagoya-u.ac.jp}
\subjclass[2000]{Primary 19K56; Secondary 46L87.}
\keywords{Coarse geometry, Dirac operator, Index theory, 
Partitioned manifolds, Relative index  theorem, Roe cocycle, 
Toeplitz operator}

\begin{document}

\begin{abstract}
In this paper, 
we extend Roe's cyclic $1$-cocycle to 
relative settings. 
We also prove two relative index theorems 
for partitioned manifolds by using its cyclic cocycle, 
which are generalizations of 
index theorems on partitioned manifolds. 
One of these 
theorems is a variant of 
\cite[Theorem 3.3]{KSZ-rp}. 
\end{abstract}

\maketitle

\section*{Introduction}

Let $M$ be a complete Riemannian manifold and 
assume that $M$ is partitioned by a closed submanifold $N$ of 
codimension $1$ into two submanifolds $M^{+}$ and $M^{-}$ 
with common boundary 
$N = M^{+} \cap M^{-} = \partial M^{+} = \partial M^{-}$. 
In this setting, 
J. Roe \cite{MR996446}  
defined a cyclic $1$-cocycle $\zeta_{N}$ 
and proved the following index theorem on partitioned manifolds. 

Let $D$ be the Dirac operator over $M$ 
and $D_{N}$ the graded Dirac operator over $N$ 
which is induced by $D$. 
In \cite{MR996446},   
Roe defined a coarse index class 
$\text{\rm c-ind}(D) = [u_{D}] \in K_{1}(C^{\ast}(M))$,  
which is a $K_{1}$-class of the Roe algebra $C^{\ast}(M)$ 
and represented by the Cayley transform $u_{D}$ of $D$. 
Roe's cyclic $1$-cocycle $\zeta_{N}$ induces 
an additive map 
$(\zeta_{N})_{\ast} : K_{1}(C^{\ast}(M)) \to \mathbb{C}$ 
by using Connes' pairing.  
By using these ingredients, 
Roe proved an index theorem on partitioned manifolds: 
\begin{equation}
\label{eq:Roe}
(\zeta_{N})_{\ast}(\text{c-ind}(D)) = 
-\frac{1}{8\pi i}
\ind (D_{N}^{+}), 
\end{equation}
here $\ind (D_{N}^{+})$ in the right hand side is the Fredholm index of $D_{N}^{+}$. 

On the other hand, 
because of the vanishing of the Fredholm index 
of the Dirac operator on closed manifolds of odd dimension, 
the value $(\zeta_{N})_{\ast}(\text{c-ind}(D))$ is trivial when 
$M$ is of even dimension. 
The author 
\cite{seto2,MR3599501}
proved an index theorem with a nontrivial value 
$(\zeta_{N})_{\ast}(x)$ for some $x \in K_{1}(C^{\ast}(M))$ when 
$M$ is of even dimension. 
The index theorem is as follows. 

Let $C_{\mathrm{w}}(M)$ be a $C^{\ast}$-algebra 
generated by bounded and smooth functions on $M$ 
with which gradient is bounded. 
A ``good'' $GL_{l}(\mathbb{C})$-valued 
continuous function $\phi \in GL_{l}(C_{\mathrm{w}}(M))$ 
defines a $K_{1}$-class $[\phi ] \in K_{1}(C_{\mathrm{w}}(M))$. 
The author defined a $KK$-class 
$[D] \in KK(C_{\mathrm{w}}(M) , C^{\ast}(M))$ for 
the graded Dirac operator $D$ and 
a coarse Toeplitz index
\[
\text{\rm c-ind}(\phi , D) 
= [\phi] \hat{\otimes}_{C_{\mathrm{w}}(M)} [D] \in K_{1}(C^{\ast}(M))
\] 
by using the Kasparov product 
\[
\hat{\otimes}_{C_{\mathrm{w}}(M)} : 
K_{1}(C_{\mathrm{w}}(M)) \times KK(C_{\mathrm{w}}(M) , C^{\ast}(M)) 
\to K_{1}(C^{\ast}(M)). 
\]  

On the other hand, 
let $\mathcal{H}$ be the 
closed subspace 
of the $L^{2}$-sections which is 
generated by the non-negative eigenvectors 
of the Dirac operator $D_{N}$ 
and $P$ the projection onto $\mathcal{H}^{l}$. 
Define the Toeplitz operator $T_{\phi|_{N}} : \mathcal{H}^{l} \to \mathcal{H}^{l}$ 
by $T_{\phi|_{N}}(s) = P\phi|_{N}s$. 
The Toeplitz opeator $T_{\phi |_{N}}$ 
is Fredholm since the values of $\phi|_{N}$ are in $GL_{l}(\mathbb{C})$. 
Then the author proved 
\begin{equation}
\label{eq:Seto}
(\zeta_{N})_{\ast}(\text{\rm c-ind}(\phi , D)) 
= 
-\frac{1}{8 \pi i}
\ind (T_{\phi|_{N}}). 
\end{equation}

In this paper, 
we generalize (\ref{eq:Roe}) and (\ref{eq:Seto}) to 
relative index settings partitioned by 
(possibly non-compact) submanifolds of codimension $1$. 
For this purpose, 
we generalize three ingredients, 
the index class in $K_{1}(C^{\ast}(M))$, 
the cyclic cocycle $\zeta_{N}$ 
and the Fredholm index on $N$ 
to the case of relative index settings, respectively. 


Let $M_{1}$ and $M_{2}$ be two complete Riemannian manifolds 
and 
$W_{1} \subset M_{1}$ and $W_{2} \subset M_{2}$ 
are closed subsets. 
Assume that 
there exists an isometry 
$\psi : M_{2} \setminus W_{2} \to M_{1} \setminus W_{1}$ 
such that 
$\psi$ conjugates all ingredients, 
for example, 
$D_{1} = (\psi^{\ast})^{-1} D_{2} \psi^{\ast}$ 
for the Dirac operators. 
Denote by 
$C^{\ast}(W_{1} \subset M_{1})$ the relative ideal 
in the Roe algebra $C^{\ast}(M_{1})$, 
which is generated by 
controlled and locally tracable operators 
which is supported near $W_{1}$. 
In this setting, 
Roe 
\cite{MR3439130}  
defined 
a coarse relative index 
$\text{\rm c-ind}(D_{1} , D_{2}) \in K_{1}(C^{\ast}(W_{1} \subset M_{1}))$ 
(see also 
\cite{KSZ-rp},  
\cite{MR1399087} 
and Definition \ref{dfn:crelind}), 
which is a generalization of Roe's odd index.  
In this paper, 
we also define 
a coarse relative Toeplitz index 
$\text{\rm c-ind}(\phi_{1} , D_{1}, \phi_{2} , D_{2}) 
\in K_{1}(C^{\ast}(W_{1} \subset M_{1}))$ 
(see Definition \ref{dfn:crelToeind}), 
which is a generalization of 
the coarse Toeplitz index 
$\text{\rm c-ind}(\phi , D)$. 
Roughly speaking, these coarse relative index classes 
are given by the difference of odd index classes 
for non-relative settings, respectively. 

We define the Roe type cyclic $1$-cocycle 
on a dense subalgebra in the relative ideal $C^{\ast}(W_{1} \subset M_{1})$ 
when $M_{i}$ is partitioned by 
a (possibly non-compact) submanifold $N_{i}$; 
see Section \ref{sec:cocycle}. 
The cyclic cocycle $\zeta$ induces an additive map 
$\zeta_{\ast} : K_{1}(C^{\ast}(W_{1} \subset M_{1})) \to \mathbb{C}$. 
In our main theorems, 
we send above coarse relative index classes by $\zeta_{\ast}$,  
then 
we get relative topological indices on $N_{i}$ 
which are introduced by  
M. Gromov and H. B. Lawson 
\cite{MR720933}.  
These are generalizations of index theorems on partitioned manifolds.  
Note that, 
Theorem 1 is a variant of 
\cite[Theorem 3.3]{KSZ-rp}.  

\begin{mthm}{\rm (see Theorem \ref{thm:main1})}
Let $(M_{i}, W_{i} , D_{i})$ be 
a tuple of 
a complete Riemannian manifold $M_{i}$ partitioned by $N_{i}$, 
a closed subset $W_{i}$ 
and 
the Dirac operator $D_{i}$ 
as previously. 
Then the following formula holds: 
\[
\zeta_{\ast}(\text{\rm c-ind}(D_{1} , D_{2})) 
= -\frac{1}{8\pi i}\text{\rm ind}_{t}(D_{N_{1}} , D_{N_{2}}) ,  
\]
here the right hand side is Gromov-Lawson's relative topological index. 
\end{mthm}

\begin{mthm}{\rm (see Theorem \ref{thm:main2})}
Let $(M_{i}, W_{i} , D_{i})$ be 
a tuple of 
a complete Riemannian manifold $M_{i}$ partitioned by $N_{i}$, 
a closed subset $W_{i}$ 
and 
the Dirac operator $D_{i}$ 
as previously. 
Take $\phi_{i} \in GL_{l}(C_{\mathrm{w}}(M_{i}))$ 
such that $\phi_{2} = \phi_{1} \circ \psi$. 
Then the following formula holds: 
\[
\zeta_{\ast}(\text{\rm c-ind}(\phi_{1} , D_{1} , \phi_{2} , D_{2})) 
= -\frac{1}{8\pi i}\text{\rm ind}_{t}(\phi_{N_{1}} , D_{N_{1}} , \phi_{N_{2}} , D_{N_{2}}) . 
\]
\end{mthm}

The strategy of the proof of the theorems is the following. 
Firstly, 
we reduce to the product case, 
which is similar to the case for index theorems on partitioned manifolds. 
Secondly, 
we prove the product case. 
In the second step, 
we use index theorems (1) and (2). 

Note that, 
in the definition of 
relative topological indices in the right hand sides, 
we use 
compactifications of neighborhoods of 
$N_{1} \cap W_{1}$ and $N_{2} \cap W_{2}$. 
However, in our proof, we do not use the fact that 
relative topological indices 
do not depend on the choice of 
such compactifications.  
Thus our main theorems give a new proof 
of well definedness of relative topological indices, respectively.

\section{Index classes}

\subsection{Relative index data}
\label{subsec:reldata}

Let $M$ be a complete Riemannian manifold  
and $W \subset M$ a closed subset. 
In this subsection, 
we recall the notion of a relative index data over a pair $(M,W)$ 
and a relative ideal in the Roe algebra. 
Coarse relative indices are elements in $K$-theory of 
its ideal. 
See 
\cite{MR3439130}
for details of these notions. 

\begin{dfn}
Let $M_{i}$ ($i=1,2$) be a complete Riemannian manifold and 
$D_{i}$ the Dirac operator 
on a Clifford bundle $S_{i} \to M_{i}$. 
We call $(M_{i}, W_{i}, D_{i})$ an odd relative index data 
over $(M,W)$
if the following holds: 
\begin{itemize} 
\item 
$W_{i} \subset M_{i}$ is a closed subset, 
\item 
there exists 
isometry $\psi : M_{2} \setminus W_{2} \to M_{1} \setminus W_{1}$ 
which induces isometry of Clifford bundles 
$\psi^{\ast} : S_{1}|_{M_{1} \setminus W_{1}} \to S_{2}|_{M_{2} \setminus W_{2}}$, 
\item 
there exists a continuous coarse map $f_{i} : M_{i} \to M$ 
such that 
$f(W_{i}) = W$, $W_{i} = f^{-1}(W)$ and 
$f_{1} \circ \psi = f_{2}$. 
\end{itemize}

An even relative index data is given by 
an odd relative index data $(M_{i}, W_{i}, D_{i})$ 
together with a $\mathbb{Z}_{2}$-grading $\epsilon_{i}$ on 
a Clifford bundle $S_{i}$ 
such that 
$D_{i}$ is the graded Dirac operator on $S_{i}$ and 
$\psi^{\ast}$ respects $\mathbb{Z}_{2}$-gradings. 
We omit odd or even when it is not important. 
\end{dfn}

Coarse relative indices are constructed by using 
a relative index data 
and 
are elements in $K$-theory of the relative ideal 
of the Roe algebra. 
Let $S \to M$ be a Hermitian vector bundle and 
recall that 
a bounded operator $T : L^{2}(M,S) \to L^{2}(M,S)$ 
is \textit{controlled} if 
there exists a constant $R > 0$ such that 
$\varphi T \psi = 0$ when $\varphi , \psi \in C_{c}(M)$ 
satisfy $d(\mathrm{Supp}(\varphi) , \mathrm{Supp}(\psi)) > R$. 
The infimum of such $R > 0$ is called 
\textit{propagation} of 
a controlled operator $T$. 
A bounded operator $T$ on $L^{2}(M,S)$ is \textit{locally tracable} 
(resp. \textit{lacally compact})
if 
$\varphi T \psi$ is of trace class (resp. compact) 
for any $\varphi , \psi \in C_{c}(M)$. 
The Roe algebra $C^{\ast}(M)$ is defined to be 
the norm closure of 
the set of controlled and locally tracable operators on $L^{2}(M,S)$. 

An operator $T$ is 
\textit{supported near $W$} if 
there exists constant $r > 0$  
such that 
$\varphi T = 0$ and $T \varphi = 0$ when 
$\varphi \in C_{c}(M)$ satisfies $d(\mathrm{Supp}(\varphi) , W) > r$. 
We call $T$ is \textit{supported in $N_{r}(W)$} 
by using such a constant $r$. 
Here, we set 
$N_{r}(A) = \{ x \in M \,;\, d(x,A) \leq r \}$ 
for a subset $A \subset M$. 
Denote by $\mathscr{B}_{W}$ 
the set of 
controlled and locally tracable operators which is supported near $W$. 
The relative ideal $C^{\ast}(W \subset M)$ 
is an ideal in $C^{\ast}(M)$ 
generated by $\mathscr{B}_{W}$. 

Let $(M_{i} , W_{i} , D_{i})$ is 
a relative index data over $(M,W)$ and 
denote by 
$\pi_{i} : C^{\ast}(M_{i}) \to C^{\ast}(M_{i})/C^{\ast}(W_{i} \subset M_{i})$ 
the projection onto the quatient. 
The isometry $\psi : M_{2} \setminus W_{2} \to M_{1} \setminus W_{1}$ 
appeared in a relative index data 
induces an isomorphism of $C^{\ast}$-algebras: 
\[
\Psi : C^{\ast}(M_{1})/C^{\ast}(W_{1} \subset M_{1}) 
	\to C^{\ast}(M_{2})/C^{\ast}(W_{2} \subset M_{2}). 
\]
As is well known, 
we have $f(D_{i}) \in C^{\ast}(M_{i})$ for any $f \in C_{0}(\mathbb{R})$. 
The isomorphism $\Psi$ gives 
a correspondence of 
$\pi_{1}(f(D_{1}))$ and $\pi_{2}(f(D_{2}))$. 

\begin{lem}
\cite[Lemma 4.3]{MR3439130}
\label{lem:C*}
Let $(M_{i} , W_{i} , D_{i})$ be a relative index data 
over $(M,W)$. 
For any $f \in C_{0}(\mathbb{R})$, we have 
\[
\Psi (\pi_{1}(f(D_{1}))) = \pi_{2}(f(D_{2})). 
\]
\end{lem}

By Lemma \ref{lem:C*}, 
a pair $(f(D_{1}), f(D_{2}))$ for $f \in C_{0}(\mathbb{R})$ 
defines an element in 
a $C^{\ast}$-algebra 
\[
\mathfrak{C} := \{ (T_{1}, T_{2}) \in C^{\ast}(M_{1}) \oplus C^{\ast}(M_{2}) 
	\,;\, \Psi (\pi_{1}(T_{1})) = \pi_{2}(T_{2}) \}.  
\]

There is a 
$D^{\ast}$-version of this discussion. 
Let $D^{\ast}(M)$ be a $C^{\ast}$-algebra 
generated by controlled and pseudolocal operators, 
here a bounded operator $T$ is pseudolocal if $[T,\varphi]$ 
is compact for any $\varphi \in C_{0}(M)$.  
The relative ideal $D^{\ast}(W \subset M)$ is an ideal in $D^{\ast}(M)$ 
which is generated by 
controlled and pseudolocal operators 
which are supported near $W$ and are 
locally compact on $M \setminus W$. 
Denote by the same letter 
$\pi_{i} : D^{\ast}(M_{i}) \to D^{\ast}(M_{i})/D^{\ast}(W_{i} \subset M_{i})$ 
the projection onto the quatient and 
$\psi$ induces an isometry of $C^{\ast}$-algebras 
\[
\Psi : D^{\ast}(M_{1})/D^{\ast}(W_{1} \subset M_{1}) 
	\to D^{\ast}(M_{2})/D^{\ast}(W_{2} \subset M_{2}).   
\]
A continuous odd fungtion $\chi : \mathbb{R} \to [-1,1]$ 
is a \textit{choping function} (or \textit{normalizing function}) if 
we have $\chi (t) \to \pm 1$ as $t \to \pm \infty$. 
Functional falculus gives an element $\chi (D_{i}) \in D^{\ast}(M_{i})$ 
and then a variant of Lemma \ref{lem:C*} is 
as follows. 

\begin{lem}
\cite[Lemma 4.4]{MR3439130}
\label{lem:D*}
Let $(M_{i} , W_{i} , D_{i})$ be a relative index data 
over $(M,W)$. 
For any chopping function $\chi$, we have 
\[
\Psi (\pi_{1}(\chi (D_{1}))) = \pi_{2}(\chi (D_{2})). 
\]
\end{lem}

By Lemma \ref{lem:D*}, 
a pair $(\chi (D_{1}), \chi (D_{2}))$ for $f \in C_{0}(\mathbb{R})$ 
defines an element in 
a $C^{\ast}$-algebra 
\[
\mathfrak{D} := \{ (T_{1}, T_{2}) \in D^{\ast}(M_{1}) \oplus D^{\ast}(M_{2}) 
	\,;\, \Psi (\pi_{1}(T_{1})) = \pi_{2}(T_{2}) \}.  
\]

\subsection{Coarse relative index}

Following 
\cite{KSZ-rp} and 
\cite{MR3439130},   
we define the coarse relative index. 
By Lemma \ref{lem:D*}, 
we have an element 
\[
\partial \left[\frac{\chi (D_{1}) + 1}{2}, \frac{\chi (D_{2}) + 1}{2}\right] 
\in K_{1}(\mathfrak{C})
\] 
for a chopping function $\chi$, 
where $\partial$ is the exponential map in the $6$-term exact sequence 
of a short exact sequence 
$0 \to \mathfrak{C} \to \mathfrak{D} \to 
\mathfrak{D}/\mathfrak{C} \to 0$. 
$K$-theory of $\mathfrak{C}$ can be decomposed as follows. 
Let $V_{i} : L^{2}(W_{1} , S_{1}) \to L^{2}(W , S)$ 
be a unitary which covers 
surjective continuous coarse map 
$f_{i}|_{W_{i}} : W_{i} \to W$. 
We can choose $V_{i}$ with  
arbitrary small propagation, 
here $V_{i}$ has propagation less than $\delta > 0$ 
if 
we have $\varphi V_{i} \psi = 0$ 
for any
$\varphi \in C_{b}(W)$ 
and  $\psi \in C_{b}(W_{i})$ 
with $d(\mathrm{Supp}(\varphi) , f_{i}(\mathrm{Supp}(\psi))) > \delta$. 
We assume that $V_{1}$ and $V_{2}$ have propagation less than $\delta / 2 > 0$. 
Define a unitary operator 
$U   : 
L^{2}(M_{1},S_{1}) \to L^{2}(M_{2},S_{2})$ by 
\[
U = V_{2}^{\ast}V_{1} \oplus \psi^{\ast} : 
L^{2}(W_{1},S_{1}) \oplus L^{2}(M_{1} \setminus W_{1},S_{1}) 
\to 
L^{2}(W_{2},S_{2}) \oplus L^{2}(M_{2} \setminus W_{2},S_{2}).  
\]
$U$ has propagation less than $\delta$ 
and 
it induces a map 
$Ad(U) : C^{\ast}(M_{1}) \to C^{\ast}(M_{2})$ 
and 
a split of a short exact sequence of $C^{\ast}$-algebras 
\[
0 \to C^{\ast}(W_{1} \subset M_{1}) \to \mathfrak{C} \to C^{\ast}(M_{2}) \to 0. 
\] 
Here, the first map is an inclusion $T_{1} \mapsto (T_{1} , 0)$, 
the second one is the projection $(T_{1} , T_{2}) \mapsto T_{2}$ 
and the split map is 
$C^{\ast}(M_{2}) \ni T_{2} \mapsto (U^{\ast}T_{2}U , T_{2}) \in \mathfrak{C}$. 
Thus we have a direct sum decomposition 
$K_{\ast}(\mathfrak{C}) = 
K_{\ast}(C^{\ast}(W_{1} \subset M_{1})) \oplus K_{\ast}(C^{\ast}(M_{2}))$.  
Denote by 
\[
q : K_{\ast}(\mathfrak{C}) \to K_{\ast}(C^{\ast}(W_{1} \subset M_{1}))
\] 
the projection onto the first summand, 
which is independent of the choice of 
a unitary $U : L^{2}(M_{1},S_{1}) \to L^{2}(M_{2},S_{2})$ 
with 
$U = \psi^{\ast}$ on $L^{2}(M_{1} \setminus W_{1} , S_{1})$
such that $U$ induces a map 
$Ad(U) : C^{\ast}(M_{1}) \to C^{\ast}(M_{2})$. 
By using the map 
$q : K_{\ast}(\mathfrak{C}) \to K_{\ast}(C^{\ast}(W_{1} \subset M_{1}))$,
we define a coarse relative index 
$\text{c-ind}(D_{1}, D_{2}) \in K_{1}(C^{\ast}(W_{1} \subset M_{1}))$. 

\begin{dfn}
\label{dfn:crelind}
Let $(M_{i},W_{i},D_{i})$ be an odd relative index data over $(M,W)$. 
The coarse relative index is defined to be 
\[
\text{\rm c-ind}(D_{1},D_{2}) = 
q\left( \partial \left[\frac{\chi (D_{1}) + 1}{2}, \frac{\chi (D_{2}) + 1}{2}\right]\right) 
	\in K_{1}(C^{\ast}(W_{1} \subset M_{1})). 
\]
\end{dfn}

Remark that the coarse relative index 
$\text{\rm c-ind}(D_{1}, D_{2}) \in K_{1}(C^{\ast}(W_{1} \subset M_{1}))$ 
is represented by using 
a continuous function $f \in U_{1}(C_{0}(\mathbb{R}))$ 
such that 
the $K_{1}$-class 
$[f]  \in K_{1}(C_{0}(\mathbb{R}))$ is a generator 
which equals $\left[\frac{x-i}{x+i} \right] \in K_{1}(C_{0}(\mathbb{R}))$ 
and the Fourier transform of $f$ 
is compactly supported 
$\mathrm{Supp}(\hat{f}) \subset (-r , r)$: 
$\text{\rm c-ind}(D_{1}, D_{2}) = 
\left[ f(D_{1}) \right] - \left[ U^{\ast} f(D_{2}) U \right]$. 
Note that an operator $f(D_{1}) - U^{\ast}f(D_{2})U$ is 
supported in $N_{r}(W_{1})$ by the proof of 
\cite[Lemma 4.3]{MR3439130}. 
We also have 
\[ 
\text{\rm c-ind}(D_{1}, D_{2}) = 
\left[ \frac{D_{1}-i}{D_{1}+i} \right] 
- \left[ U^{\ast} \frac{D_{2}-i}{D_{2}+i} U \right] 
\in K_{1}(C^{\ast}(W_{1} \subset M_{1}).  
\]

Similarly, an even relative index data defines an 
even index class in $K_{0}(C^{\ast}(W_{1} \subset M_{1}))$; 
see 
\cite{MR3439130}. 
We do not use the even class in this paper.

\subsection{Coarse relative Toeplitz index}

Let $(M_{i},W_{i},D_{i})$ be an even relative index data over $(M,W)$. 
We define a coarse relative Toeplitz index 
$\text{\rm c-ind}(\phi_{1} , D_{1} , \phi_{2}, D_{2}) 
\in K_{1}(C^{\ast}(W_{1} \subset M_{1}))$ 
of $(M_{i},W_{i},D_{i})$ 
and a function $\phi_{i}$ on $M_{i}$. 

Let $C_{\mathrm{w}}(M)$ be a $C^{\ast}$-algebra 
generated by  $\mathscr{W} = \mathscr{W}(M)$, 
which is the set of smooth and bounded functions 
with which 
gradient is bounded; see 
\cite[Definition 2.1]{seto2}. 
We define a relative version of this $C^{\ast}$-algebra. 
Denote by $\mathfrak{W}$ 
a $C^{\ast}$-algebra generated by 
$(f_{1}, f_{2}) \in \mathscr{W}(M_{1}) \oplus \mathscr{W}(M_{2})$ 
such that $f_{1} \circ \psi = f_{2}$ 
on the complement of  $W_{2}$, then we have  
\[
\mathfrak{W} 
= 
\{
(f_{1} , f_{2}) \in C_{\mathrm{w}}(M_{1}) \oplus C_{\mathrm{w}}(M_{2}) \,;\, 
f_{1} \circ \psi = f_{2} 
\}. 
\]

Let $\chi$ be a chopping function and set $\eta = (1-\chi^{2})^{1/2}$ and 
$\mathcal{D}_{i} = \chi (D_{i}) + \epsilon \eta (D_{i}) \in D^{\ast}(M_{i})$. 
Similar to 
\cite[Proposition 2.3]{seto2},  
an even relative index data 
$(M_{i}, W_{i}, D_{i})$ defines a 
$KK$-element 
$[\mathfrak{C} , (\chi(D_{1}), \chi(D_{2}))] 
\in KK(\mathfrak{W}, \mathfrak{C})$. 
This class does not depends on the choice of 
a chopping function $\chi$. 

Any $(\phi_{1} , \phi_{2}) \in GL_{l}(\mathfrak{W})$ determines 
a $K_{1}$-class 
$[\phi_{1} , \phi_{2}] \in K_{1}(\mathfrak{W})$. 
By using the Kasparov product 
\[
\hat{\otimes}_{\mathfrak{W}} : 
K_{1}(\mathfrak{W}) \times KK(\mathfrak{W} , \mathfrak{C}) 
\to K_{1}(\mathfrak{C}), 
\] 
we have an element 
$[\phi_{1} , \phi_{2}] \hat{\otimes}_{\mathfrak{W}} 
[\mathfrak{C} , (\chi(D_{1}), \chi(D_{2}))] 
\in K_{1}(\mathfrak{C})$. 

Set 
\[ 
u_{\phi_{i}} = 
\mathcal{D}_{i} 
\begin{bmatrix} \phi_{i} & 0 \\ 0 & 1 \end{bmatrix}
\mathcal{D}_{i} 
\begin{bmatrix} 1 & 0 \\ 0 & \phi_{i}^{-1} \end{bmatrix}. 
\]
Due to Lemma \ref{lem:C*}, \ref{lem:D*} 
and 
\cite[Remark 4.2]{seto2},  
we have $(u_{\phi_{1}}, u_{\phi_{2}}) \in GL_{l}(\mathfrak{C})$. 
By the same calculation of the proof of  
\cite[Proposition 4.3]{seto2},  
we have 
\[
[\phi_{1} , \phi_{2}] \hat{\otimes}_{\mathfrak{W}} 
[\mathfrak{C} , (\chi(D_{1}), \chi(D_{2}))]  = 
[u_{\phi_{1}} , u_{\phi_{2}}] 
\in K_{1}(\mathfrak{C}). 
\]
Note that 
an operator $u_{\phi_{1}} - U^{\ast}u_{\phi_{2}}U$ is 
supported in $N_{2r}(W_{1})$ 
when 
the Fourier transform of $\chi$ is compaclty supported 
$\mathrm{Supp}(\hat{\chi}) \subset (-r , r)$ by the proof of 
\cite[Lemma 4.3]{MR3439130}. 
By using the map 
$q : K_{\ast}(\mathfrak{C}) \to K_{\ast}(C^{\ast}(W_{1} \subset M_{1}))$,
we define a coarse relative Toeplitz index as follows. 

\begin{dfn}
\label{dfn:crelToeind}
Let $(M_{i},W_{i},D_{i})$ be an even relative index data over $(M,W)$ and 
$(\phi_{1} , \phi_{2}) \in GL_{l}(\mathfrak{W})$. 
The coarse relative Toeplitz index is defined to be 
\[
\text{\rm c-ind}(\phi_{1} , D_{1}, \phi_{2} , D_{2}) 
	= q([\phi_{1} , \phi_{2}] \hat{\otimes}_{\mathfrak{W}} 
	[\mathfrak{C} , (\chi(D_{1}), \chi(D_{2}))] ) 
	\in K_{1}(C^{\ast}(W_{1} \subset M_{1})). 
\]
\end{dfn}

Similarly, an odd relative index data defines an 
even Toeplitz index class in $K_{0}(C^{\ast}(W_{1} \subset M_{1}))$. 
We do not use the even class in this paper.

\section{The Roe type cyclic $1$-cocycle in the relative setting}
\label{sec:cocycle}

Roe \cite{MR996446} 
defined a cyclic $1$-cocycle on a complete Riemannian manifold 
partitioned by a \textit{closed} submanifold of codimension $1$. 
In this section, 
we generalize the cocycle to a pair $(M,W)$ 
partitioned by a (possibly non-compact) submanifold of codimension $1$. 

\subsection{Definition of cyclic cocycle}

Let $(M,W)$ be a pair of a complete Riemannian manifold $M$ 
and 
a closed subset $W \subset M$ 
and 
$S \to M$ a Hermitian vector bundle. 
In this subsection, 
we define a cyclic $1$-cocycle on a 
dense subalgebra of a relative ideal $C^{\ast}(W \subset M)$, 
which is a generalization of the Roe cocycle. 

\begin{dfn}
Let $M$ be a complete Riemannian manifold and 
$W \subset M$ a closed subset. 
Assume that the triple $(M^{+}, M^{-}, N)$ satisfies the following conditions: 
\begin{itemize}
\item $M^{+}$ and $M^{-}$ are submanifolds of $M$ of the same dimension as $M$, 
	$\partial M^{+} \neq \emptyset$ and $\partial M^{-} \neq \emptyset$, 
\item $M = M^{+} \cup M^{-}$, 
\item $N$ is a submanifold of $M$ of codimension $1$, 
\item $N = M^{+} \cap M^{-} = -\partial M^{+} = \partial M^{-}$, 
\item $Z = N \cap W$ is compact,  
\item $N$ and $W$ are coarsely transversal, that is, 
	for any $r > 0$ there exists $s > 0$ such that 
	$N_{r}(N) \cap N_{r}(W) \subset N_{s}(Z)$. 
\end{itemize}
Then we call $(M^{+}, M^{-}, N)$ a partition of $(M,W)$. 
We also say $(M,W)$ is partitioned by $(M^{+}, M^{-}, N)$, 
or is partitioned by $N$, for short.  
\end{dfn}

Assume that $(M,W)$ is partitioned by $N$ 
and set $W^{\pm} = M^{\pm} \cap W$. 
Then for any $r > 0$, 
there exists $s > 0$ such that 
$N_{r}(W^{+}) \cap N_{r}(W^{-}) 
\subset 
N_{r}(M^{+}) \cap N_{r}(W) \cap N_{r}(M^{-}) 
\subset 
N_{r}(N) \cap N_{r}(W)
\subset 
N_{s}(Z)$. 

In order to generalize Roe's cyclic $1$-cocycle, 
we firstly prove the following. 
Denote by $\Pi$ the characteristic function of $M^{+}$ 
and set $\Lambda = 2\Pi - 1$. 

\begin{lem}
\label{lem:cycwell}
An operator $[\Lambda , A]$ is of trace class 
on $L^{2}(M,S)$ for any $A \in \mathscr{B}_{W}$. 
\end{lem}

\begin{proof}
Assume that propagation of $A \in \mathscr{B}_{W}$ is less than $R$ 
and $A$ is supported in $N_{r}(W)$. 
Take $s > 0$ such that 
$N_{r+R}(W^{+}) \cap N_{r+R}(W^{-}) \subset N_{s}(Z)$. 
Note that operators 
$\Pi A (1 - \Pi)$ 
and 
$(1 - \Pi) A \Pi$ are supported in 
$N_{r+R}(W^{+}) \cap N_{r+R}(W^{-}) \subset N_{s}(Z)$. 
Since $A$ is locally tracable and $N_{s}(Z)$ is compact, 
these operators $\Pi A (1 - \Pi)$ and $(1 - \Pi) A \Pi$ are 
of trace class. 
Thus an operator 
\[
[\Lambda , A] = 
[2\Pi , A] = 2(\Pi A (1-\Pi) - (1 - \Pi)A\Pi)
\]
is of trace class. 
\end{proof}

By Lemma \ref{lem:cycwell}, 
the following bilinear map 
is well-defined and 
is cyclic $1$-cocycle. 

\begin{dfn}
\label{dfn:relRoe}
Define a map $\zeta : \mathscr{B}_{W} \times \mathscr{B}_{W} \to \mathbb{C}$ 
by 
\[
\zeta (A,B) = \frac{1}{4}\mathrm{Tr}(\Lambda [\Lambda , A][\Lambda , B]). 
\]
\end{dfn}

\begin{prp}
The bilinear map 
$\zeta : \mathscr{B}_{W} \times \mathscr{B}_{W} \to \mathbb{C}$ 
in Definition \ref{dfn:relRoe} 
is cyclic $1$-cocycle on $\mathscr{B}_{W}$. 
\end{prp}

\begin{proof}
By equalities 
$\Lambda [\Lambda , A] = - [\Lambda , A]\Lambda$ 
and 
$[\Lambda , AB] = A[\Lambda , B] + [\Lambda , A]B$ 
and trace property imply this proposition. 
This is essentially the same as 
Roe's proof of \cite[Proposition 1.6]{MR996446}. 
\end{proof}

By Lemma \ref{lem:cycwell}, 
a pair $(L^{2}(M,S) , \Lambda)$ is 
a Fredholm module over $C^{\ast}(W \subset M) = \overline{\mathscr{B}_{W}}$. 
Thus a Banach algebra 
\[
\mathscr{A}_{W} = 
\{
T \in C^{\ast}(W \subset M) \,;\, 
[\Lambda , T] \text{ is of trace class.}
\}
\]
with norm $\| T \|_{\mathscr{A}_{W}} = \| T \| + \| [\Lambda , T] \|_{1}$ 
is holomorphically closed in $C^{\ast}(W \subset M)$ 
by 
\cite[p.92 Proposition 3]{MR823176},  
here $\| \cdot \|_{1}$ is the trace norm. 
Moreover, 
$\mathscr{A}_{W}$ is dense in $C^{\ast}(W \subset M)$. 
Thus the inclusion $i : \mathscr{A}_{W} \to C^{\ast}(W \subset M)$ 
induces an isomorphism of $K$-theory 
$i_{\ast} : K_{\ast}(\mathscr{A}_{W}) \cong K_{\ast}(C^{\ast}(W \subset M))$. 
Since
$\zeta$ can be extended to $\mathscr{A}_{W}$, 
we have the following additive map 
by Connes' pairing of $K$-theory with cyclic cohomology. 

\begin{dfn}
\cite[p.109]{MR823176} 
Define the map 
\[ \zeta_{\ast} : K_{1}(C^{\ast}(W \subset M)) \to \mathbb{C} \]
by $\zeta_{\ast} ([u]) = \frac{1}{8\pi i}\sum_{i,j}\zeta ((u^{-1})_{ji}, u_{ij})$, 
where we assume $[u]$ is represented by an element 
$u \in GL_{l}(\mathscr{A}_{W})$
and $u_{ij}$ is the $(i,j)$-component of $u$. 
We note that this is Connes' pairing of cyclic cohomology with $K$-theory, 
and $1/8\pi i$ is a constant multiple in Connes' pairing. 
\end{dfn}

By standard calculation 
implies the following; see, for instance,  
\cite[Proposition 1.13]{seto2}.  

\begin{prp}
\label{prp:paind}
For any $u \in GL_{l}(C^{\ast}(W \subset M))$, one has 
\[ \zeta_{\ast}([u])  
	= -\frac{1}{8 \pi i}\ind (\Pi u \Pi : \Pi (L^{2}(M,S))^{l} \to \Pi (L^{2}(M,S))^{l}). \]
\end{prp}

\begin{rmk}
\label{rmk:K1supp}
By an isomorphism \cite[Proposition 4.3.12]{MR3153339}
\[
K_{\ast}(C^{\ast}(W)) \cong K_{\ast}(C^{\ast}(W \subset M))
\] 
induced by the inclusion $W \to M$, 
any $x \in K_{1}(C^{\ast}(W \subset M))$ 
can be represented by $u \in GL_{l}(C^{\ast}(W \subset M))$ 
such that an operator $u - 1$ is supported in $N_{r}(W)$ 
for arbitrary small $r > 0$. 
Denote by $p_{r}$ the characteristic function of $N_{r}(W)$. 
Under the notations, we have 
\[
\zeta_{\ast}(x) = 
-\frac{1}{8 \pi i}
\ind (\Pi p_{r} u p_{r} \Pi 
: p_{r} \Pi (L^{2}(M,S))^{l} \to p_{r} \Pi (L^{2}(M,S))^{l}). 
\]
\end{rmk}

\subsection{Some properties}

Let $(M,W)$ be a pair of a complete Riemannian manifold $M$ 
and 
a closed subset $W \subset M$ 
and 
$S \to M$ a Hermitian vector bundle. 
In this subsection, 
we prove properties of the map $\zeta_{\ast}$ 
which we use. 
Firstly, 
we prove ``cobordism'' invariance of $\zeta_{\ast}$. 

\begin{lem}
\label{lem:cobor}
Let $N$ and $N'$ be two partitions of $(M,W)$. 
Denote by $\zeta$ and $\zeta'$ 
the cyclic cocycle introduced in Definition \ref{dfn:relRoe} 
by using the partitions $N$ and $N'$, respectively. 
Assume that $N' \subset N_{r}(N)$ for some $r>0$. 
Then we have 
\[
\zeta_{\ast} = \zeta'_{\ast} : K_{1}(C^{\ast}(W \subset M)) \to \mathbb{C}. 
\]
\end{lem}

\begin{proof}
Denote by $\Pi'$ the characteristic function of $M^{+}{}'$ 
and set $\phi = \Pi - \Pi'$. 
Take any $x \in K_{1}(C^{\ast}(W \subset M))$. 
$x$ can be represented  by $u = v+1 \in GL_{l}(C^{\ast}(W \subset M))$ 
such that $v \in M_{l}(C^{\ast}(W \subset M))$ 
is supported near $W$ as in Remark \ref{rmk:K1supp}. 
Then we have 
\[
\zeta_{\ast}(x) = -\frac{1}{8\pi i}\ind (\Pi u \Pi \text{ on } \Pi (L^{2}(M,S))^{l}) 
= -\frac{1}{8\pi i}\ind (\Pi v + 1)
\]
and 
\[
\zeta'_{\ast}(x) = -\frac{1}{8\pi i}\ind (\Pi' u \Pi' \text{ on } \Pi' (L^{2}(M,S))^{l}) 
= -\frac{1}{8\pi i}\ind (\Pi' v + 1)   
\]
since operators $[\Pi , v]$ and $[\Pi' , v]$ are compact. 
By the way, properties 
$\mathrm{Supp}(\phi) \subset N_{r}(N)$, 
$v$ is locally compact 
and $v$ is supported near $W$ 
imply an operator 
$(\Pi v + 1) - (\Pi'v + 1) = \phi v$ 
is compact. 
Thus we have 
$\zeta_{\ast}(x) = \zeta'_{\ast}(x)$ for any $x \in K_{1}(C^{\ast}(W \subset M))$. 
\end{proof}

Secondly, we shall prove an analogue of Higson's Lemma  
\cite[Lemma 3.1]{MR1113688}. 

\begin{lem}
\label{lem:Higson}
Let $(M , W)$ and $(M' , W')$ 
be two pairs which is partitioned by $N$ and $N'$, respectively, 
and 
$S \to M$ and $S' \to M'$ two Hermitian vector bundles. 
Let $\Pi$ and $\Pi'$ be the characteristic function of $M^{+}$ and 
$M'^{+}$, respectively. 
We assume that there exists an isometry 
$\gamma : M'^{+} \to M^{+}$ which lifts
an isomorphism $\gamma^{\ast} : S|_{M^{+}} \to S'|_{M'^{+}}$. 
We denote the Hilbert space isometry defined by $\gamma^{\ast}$ 
by the same letter $\gamma^{\ast} : \Pi(L^{2}(M, S)) \to \Pi'(L^{2}(M', S'))$. 
Take $u \in GL_{l}(C^{\ast}(W \subset M))$ and 
$u' \in GL_{l}(C^{\ast}(W' \subset M'))$ 
such that 
$\gamma^{\ast}u\Pi \sim \Pi'u'\gamma^{\ast}$. 
Then one has 
$\zeta_{\ast}([u]) = \zeta'_{\ast}([u'])$. 

Similarly, if there exists an isometry $\gamma : M'^{-} \to M^{-}$ which lifts 
an isomorphism $\gamma^{\ast} : S'|_{M'^{-}} \to S|_{M^{-}}$ and 
$\gamma^{\ast}u\Pi \sim \Pi'u'\gamma^{\ast}$, 
then one has 
$\zeta_{\ast}([u]) = \zeta'_{\ast}([u'])$. 
\end{lem}

\begin{proof}
It suffices to show the case when $l=1$. 
Let $v : (1-\Pi)(L^{2}(M, S)) \to (1-\Pi')(L^{2}(M', S'))$ 
be any invertible operator. 
Then $V = \gamma^{\ast}\Pi + v(1-\Pi) 
: L^{2}(M , S) \to L^{2}(M' , S')$ is also invertible. 
Hence we obtain 
\begin{align*}
& \; V((1-\Pi) + \Pi u\Pi) - ((1-\Pi') + \Pi ' u'\Pi')V \\
=& \; \gamma^{\ast}\Pi u \Pi - \Pi' u' \Pi' \gamma^{\ast} 
\sim  \gamma^{\ast}u \Pi - \Pi' u \gamma^{\ast} \sim 0. 
\end{align*}
Therefore, we obtain $\zeta_{\ast}([u]) = \zeta'_{\ast}([u'])$ 
since $V$ is an invertible operator and one has 
$-8\pi i \zeta_{\ast}([u]) = 
\ind (\Pi u\Pi) = \ind ((1-\Pi) + \Pi u \Pi )$ 
and 
$-8\pi i \zeta'_{\ast}([u']) = 
\ind ((1-\Pi') + \Pi' u' \Pi' )$.  
\end{proof}

Let $(M , W)$ and $(M' , W')$ 
be two pairs which is partitioned by $N$ and $N'$, respectively, 
and 
$S \to M$ and $S' \to M'$ two Hermitian vector bundles. 
Let $\Pi$ and $\Pi'$ be the characteristic function of $M^{+}$ and 
$M'^{+}$, respectively. 
We assume that there exists an isometry 
$\gamma : N_{r}(W') \to N_{r}(W)$ 
which lifts an isomorphism 
$\gamma^{\ast} : S_{1}|_{N_{r}(W)} \to S_{2}|_{N_{r}(W')}$. 
We define an additive map 
$\Gamma : K_{1}(C^{\ast}(W \subset M)) 
\to K_{1}(C^{\ast}(W' \subset M'))$ 
as follows. 
Take any $x \in K_{1}(C^{\ast}(W \subset M))$. 
$x$ can be represented  by $u = v+1 \in GL_{l}(C^{\ast}(W \subset M))$ 
such that $v \in M_{l}(C^{\ast}(W \subset M))$ 
is supported in $N_{r}(W)$; see Remark \ref{rmk:K1supp}. 
Then we have 
$\gamma^{\ast}u(\gamma^{\ast})^{-1} 
= \gamma^{\ast}v(\gamma^{\ast})^{-1} + 1 \in GL_{l}(C^{\ast}(W' \subset M'))$, 
here 
$\gamma^{\ast}v(\gamma^{\ast})^{-1}  \in M_{l}(C^{\ast}(W' \subset M'))$ 
is supported in $N_{r}(W')$. 
The $K_{1}$-class of $\gamma^{\ast}u(\gamma^{\ast})^{-1} $ 
does not depend on the choice of such an above $u$. 
Set 
$\Gamma (x) 
= [\gamma^{\ast}u(\gamma^{\ast})^{-1}] \in K_{1}(C^{\ast}(W' \subset M'))$. 

\begin{lem}
\label{lem:locW}
Moreover, we assume that an isometry $\gamma : N_{r}(W') \to N_{r}(W)$ 
preserves partitions, that is, $\gamma$ satisfies 
$\Pi' = \Pi \circ \gamma$ on $N_{r}(W')$. 
Then we have 
$\zeta'_{\ast} \circ \Gamma = \zeta_{\ast}$ on $K_{1}(C^{\ast}(W \subset M))$. 
\end{lem}

\begin{proof}
Let $p_{r}$ and $p'_{r}$ be the characteristic function 
of $N_{r}(W)$ and $N_{r}(W')$, respectively. 
Take any $x \in K_{1}(C^{\ast}(W \subset M))$. 
We represent $x$ by $u = v+1 \in GL_{l}(C^{\ast}(W \subset M))$ 
such that $v \in M_{l}(C^{\ast}(W \subset M))$ 
is supported in $N_{r}(W)$. 
We have 
\[
\zeta'_{\ast} \circ \Gamma (x) 
= 
\ind \left( (1 - p'_{r} \Pi') + \Pi' p'_{r} \gamma^{\ast}u(\gamma^{\ast})^{-1} p'_{r} \Pi' \right).  
\]
Take any invertible operator  
$\nu : 
L^{2}((N_{r}(W) \cap M^{+})^{c} , S) \to L^{2}((N_{r}(W') \cap M'^{+})^{c} , S')$ 
and set $V = \gamma^{\ast} + \nu$. 
Then we have 
\begin{align*}
& \; \left( (1 - p'_{r} \Pi') + \Pi' p'_{r} \gamma^{\ast}u(\gamma^{\ast})^{-1} p'_{r} \Pi' \right) V - V\left((1-p_{r}\Pi) + \Pi p_{r} u p_{r}\Pi \right) \\ 
=& \; \Pi' p'_{r} \gamma^{\ast}u(\gamma^{\ast})^{-1} p'_{r} \Pi' \gamma^{\ast} 
	- \gamma^{\ast} \Pi p_{r} u p_{r}\Pi 
= 0, 
\end{align*}
here we used $\Pi' = \Pi \circ \gamma$ on $N_{r}(W')$. 
Thus we obtain 
\begin{align*}
\zeta'_{\ast} \circ \Gamma (x) &= 
-\frac{1}{8\pi i} \ind \left( (1 - p'_{r} \Pi') + \Pi' p'_{r} \gamma^{\ast}u(\gamma^{\ast})^{-1} p'_{r} \Pi' \right) \\ 
&= -\frac{1}{8\pi i}\ind \left((1-p_{r}\Pi) + \Pi p_{r} u p_{r}\Pi \right) 
= \zeta_{\ast}(x). 
\end{align*}
\end{proof}

\section{Relative index formula on odd dimension}

In this section, 
we state and prove an index theorem for 
an odd relative index data partitioned by 
submanifolds of codimension $1$. 
This index formula is a variant of 
\cite[Theorem 3.3]{KSZ-rp}. 

\subsection{Index theorem}
\label{subsec:odd}

Firstly, we introduce a partition of a relative index data 
over a pair $(M,W)$. 
See also 
\cite{KSZ-rp}. 

\begin{dfn}
Let 
$(M,W)$ be a pair of a complete Riemannian manifold $M$ and 
$W \subset M$ a closed subset.  
We say a relative index data 
$(M_{i} , W_{i} , D_{i})$ over $(M,W)$ 
is partitioned by $(N_{1}, N_{2})$ if the following hold. 
\begin{itemize}
\item 
$(M_{i} , W_{i})$ is partitioned by $N_{i}$, 
\item 
$\Pi_{2} = \Pi_{1} \circ \psi$, 
here $\Pi_{i}$ is the characteristic function of $M_{i}^{+}$, 
\item 
there exists a closed submanifold $N \subset M$ 
which partitions $(M,W)$ such that 
$N_{i} = f_{i}^{-1}(N)$ and $f_{i}(Z_{i}) = W \cap N$. 
\end{itemize}
\end{dfn}

Let $(M_{i} , W_{i} , D_{i})$ be an odd relative index data 
over $(M,W)$ 
partitioned by $(N_{1} , N_{2})$. 
The Dirac operator $D_{i}$ induces 
a graded Dirac operator $D_{N_{i}}$ on $S|_{N_{i}} \to N_{i}$ 
and 
they satisfies 
$(\psi_{N_{2} \setminus Z_{2}})^{\ast} \circ D_{N_{1}} 
= D_{N_{2}} \circ (\psi_{N_{2} \setminus Z_{2}})^{\ast}$. 
Then 
the relative topological index 
$\text{\rm ind}_{t}(D_{N_{1}} , D_{N_{2}}) \in \mathbb{Z}$ 
is obtained. 

Following 
\cite[Section 4]{MR720933}, 
we recall the definition of 
the relative topological index
$\text{\rm ind}_{t}(D_{N_{1}} , D_{N_{2}}) \in \mathbb{Z}$. 
Chop off the manifold $N_{i}$ outside of $Z_{i}$ 
by a closed submanifold $H_{i}$ of codimension $1$ 
to obtain a compact manifold $\Omega_{i}$ with boundary 
$\partial \Omega_{i} = H_{i}$ such that 
$N_{r}(Z_{i}) \subset \Omega_{i}$ for some $r > 0$. 
Let $\widetilde{N_{i}}$ 
be a closed manifold such that 
$\Omega_{i} \subset \widetilde{N_{i}}$
and $\widetilde{S_{N_{i}}} \to \widetilde{N_{i}}$ a graded Clifford bundle. 
Assume that 
all structures on 
$\Omega_{i} \subset N_{i}$ are isomorphic to those on  
$\Omega_{i} \subset \widetilde{N_{i}}$, respectively, 
and there exists an isometry 
$\widetilde{\psi} : \widetilde{N_{2}} \setminus Z_{2} 
	\to \widetilde{N_{1}} \setminus Z_{1}$ 
such that $\widetilde{\psi}$ induces an isometry of graded Clifford bundles 
$\widetilde{\psi}^{\ast} : 
\widetilde{S_{N_{1}}}|_{\widetilde{N_{1}} \setminus Z_{1}} 
	\to \widetilde{S_{N_{2}}}|_{\widetilde{N_{2}} \setminus Z_{2}}$. 
There is the graded Dirac opeator $\widetilde{D_{N_{i}}}$ on 
$\widetilde{S_{N_{i}}}$. 
Set 
\[
\text{\rm ind}_{t}(D_{N_{1}} , D_{N_{2}}) 
= 
\ind (\widetilde{D_{N_{1}}}^{+}) - \ind (\widetilde{D_{N_{2}}}^{+}) \in \mathbb{Z}. 
\]
The value is independent of the choice of 
a compactification $\widetilde{N_{i}}$ 
and a graded Clifford bundle $\widetilde{S_{i}}$. 

Our first main theorem is the following. 
This is a variant of 
\cite[Theorem 3.3]{KSZ-rp}. 

\begin{thm}
\label{thm:main1}
Let $M$ be a complete Riemannian manifold and 
$W \subset M$ a closed subset. 
Let $(M_{i}, W_{i} , D_{i})$ be an odd 
relative index data over $(M,W)$ 
which is partitioned by $(N_{1} , N_{2})$. 
Then the following formula holds: 
\[
\zeta_{\ast}(\text{\rm c-ind}(D_{1} , D_{2})) 
= -\frac{1}{8\pi i}\text{\rm ind}_{t}(D_{N_{1}} , D_{N_{2}}) . 
\]
\end{thm}

We prove Theorem \ref{thm:main1} in Subsection \ref{subsec:pr1}. 
In the proof of Theorem \ref{thm:main1}, 
we do not use the fact that 
the relative topological index 
$\text{\rm ind}_{t}(D_{N_{1}} , D_{N_{2}})$ 
does not depend on the choice of 
a compactification $\widetilde{N_{i}}$ 
and a graded Clifford bundle $\widetilde{S_{N_{i}}}$.   
Thus Theorem \ref{thm:main1} gives a new proof 
of well definedness of 
$\text{\rm ind}_{t}(D_{N_{1}} , D_{N_{2}})$. 

\begin{rmk}
By the definition of the relative topological index and 
the vanishing of the Fredholm index of the Dirac operator 
on closed manifolds of odd dimension, 
the relative topological index vanishes when $M_{i}$ is of even dimension. 
Thus the value $\zeta_{\ast}(\text{\rm c-ind}(D_{1} , D_{2}))$ also vanishes. 
We prove another relative index theorem for partitioned manifolds 
with non-vanishing the value $\zeta_{\ast}(x)$ 
when $M_{i}$ is of even dimension,  
in Section \ref{sec:even}. 
\end{rmk}

\subsection{Proof}
\label{subsec:pr1}

In this subsection, 
we prove Theorem \ref{thm:main1}. 
There are $2$ steps to prove it, 
the first one is the reduction to the product case and 
the second one is the proof of the product case. 

Firstly, we reduce the product case. 
Take a tubular neighbourhood of $N_{i}$ 
diffeomorphic to $(-1, 1) \times N_{i}$ such that 
$[0, 1) \times N_{i} \subset M_{i}^{+}$. 
By Lemma \ref{lem:Higson}, 
we can replace $W_{i}$ by $W_{i} \cup (Z_{i} \times [0, 1))$ 
without changing the value $\zeta_{\ast}(\text{\rm c-ind}(D_{1} , D_{2}))$. 
Fix small $r > 0$. 
Take a submanifold $N'_{i} \subset M_{i}$ 
which partitions $(M_{i} , W_{i})$ such that 
$N'_{i} \subset N_{r}(M_{i}^{-})$ and 
$N'_{i} \cap M_{i}^{-} = \emptyset$. 
Denote by $\zeta_{r}$ the cyclic cocycle 
defined by using this new partition. 
By Lemma \ref{lem:cobor}, 
we have 
\[
\zeta_{\ast}(\text{\rm c-ind}(D_{1} , D_{2}))  
= (\zeta_{r})_{\ast}(\text{\rm c-ind}(D_{1} , D_{2})).  
\]

Next we take 
a relative index data 
$(M_{i}' , W_{i}' , D_{i}')$ over $(M' , W')$ 
as follows. 
We set  
$M_{\bullet}' 
= (\mathbb{R}_{-} \times N_{\bullet}) 
	\cup M_{\bullet}^{+}$ and    
$W_{\bullet}' 
= (\mathbb{R}_{-} \times Z_{\bullet}) 
	\cup W_{\bullet}^{+}$,  
here $\bullet = 1, 2$ or empty and 
a metirc is product 
on $(-\infty , -r] \times N_{\bullet}$. 
A Clifford bundle $S_{i}' \to M_{i}'$ 
satisfies 
$S_{i}'|_{M_{i}^{+}} = S_{i}|_{M_{i}^{+}}$ 
and $S_{i}' = (-\infty , -r] \times S_{i}|_{N_{i}}$.  
An isometry 
$\psi' : M_{2}' \setminus W_{2}' \to M_{1}' \setminus W_{1}'$ 
satisfies $\psi' |_{M_{2}^{+}} = \psi |_{M_{2}^{+}}$
and $\psi'  = \mathrm{id} \times \psi |_{N}$ 
on $(-\infty , -r] \times (Z_{2})^{c}$ 
and 
a continuous coarse map 
$f_{i} : M_{1}' \to M'$ 
satisfies 
$f_{i}' |_{M_{i}^{+}} = f_{i} |_{M_{i}^{+}}$ 
and 
$f_{i}' = \mathrm{id} \times f_{i} |_{N}$ 
on $(-\infty , -r] \times N_{2}$. 
A pair $(M_{1}' , W_{1}')$ is partitioned by $N_{1}'$. 
Denote by $\zeta'$ the cyclic cocycle on $M_{1}'$. 
Note that 
we do not have to care 
the relative index data $(M_{i}' , W_{i}' , D_{i}')$ 
is partitioned or not, that is, 
we do not have to care 
$N_{i}' = (f_{i}')^{-1}(N')$ holds or not 
for some $N'$. 

\begin{lem}
\label{lem:red1}
We have 
\[
(\zeta_{r})_{\ast}(\text{\rm c-ind}(D_{1} , D_{2})) 
= 
\zeta'_{\ast}(\text{\rm c-ind}(D_{1}' , D_{2}')). 
\]
\end{lem}

\begin{proof}
Take unitaries $U$ and $U'$ 
appeared in the definition of the projection $q$ 
such that 
$U = U'$ on $L^{2}(M_{1}^{+} , S_{1}')$ and 
propagation of $U$ and $U'$ is less than $r/4$. 
Take a function $f \in U_{1}(C_{0}(\mathbb{R}))$ such that 
$\mathrm{Supp}(\hat{f}) \subset (-r/4 , r/4)$ and 
$[f] = \left[ \frac{x-i}{x+i} \right] \in K_{1}(C_{0}(\mathbb{R}))$. 
We have 
\[
\text{\rm c-ind}(D_{1}, D_{2}) = 
\left[ f(D_{1})U^{\ast}f(D_{2})^{\ast}U \right]  
\]
and 
\[
\text{\rm c-ind}(D_{1}', D_{2}') = 
\left[ f(D_{1}')U'^{\ast}f(D_{2}')^{\ast}U' \right] .  
\]
Since propagation of 
$f(D_{1})U^{\ast}f(D_{2})^{\ast}U$ 
and 
$f(D_{1}')U'^{\ast}f(D_{2}')^{\ast}U'$ 
is 
less than $r$, 
we have 
\[
\Pi_{1,r} f(D_{1})U^{\ast}f(D_{2})^{\ast}U \Pi_{1,r} 
= 
\Pi_{1}' f(D_{1}')U'^{\ast}f(D_{2}')^{\ast}U' \Pi_{1}', 
\]
here $\Pi_{1,r}$ (resp. $\Pi_{1}'$) 
is the characteristic function of 
$M_{1,r}^{+}$ (resp. $M_{1}'{}^{+} (= M_{1,r}^{+})$). 
By using Lemma \ref{lem:Higson}, 
we complete the proof.

%
\end{proof}

We apply the same argument to $(M_{i}', W_{i}', D_{i}')$, 
so that 
the proof is reduced to the following product case. 
Let $(N,Z)$ be a pair of 
a complete Riemannian manifld $N$ and 
a compact subset $Z \subset N$ and 
$(N_{i} , Z_{i} , D_{N_{i}})$ 
an even relative index data over $(N,Z)$. 
Then 
$(M_{i} = \mathbb{R} \times N_{i} , 
W_{i} = \mathbb{R} \times Z_{i} , D_{i})$ 
is an odd relative index data over 
$(M = \mathbb{R} \times N , W = \mathbb{R} \times Z)$, 
here the Dirac operator $D_{i}$ on $\mathbb{R} \times N_{i}$ 
is canonically defined by using $D_{N_{i}}$.  
$(M_{i} , W_{i} , D_{i})$ is partitioned by 
$(\{ 0 \} \times N_{1} , \{ 0 \} \times N_{2})$. 
An isometry $\psi$ and a continuous coarse map $f_{i}$ 
are given by the product 
$\mathrm{id}_{\mathbb{R}} \times \psi_{N}$ 
and $\mathrm{id}_{\mathbb{R}} \times f_{N_{i}}$, respectively. 

Let us prove the product case. 
Take closed manifolds $\widetilde{N_{1}}$ and $\widetilde{N_{2}}$ 
as in the definition of the relative topological index. 
Set 
$\widetilde{M_{i}} = \mathbb{R} \times \widetilde{N_{i}}$ 
and 
$\widetilde{W_{i}} = \mathbb{R} \times Z_{i} = W_{i}$. 
Then 
$(\widetilde{M_{i}} , \widetilde{W_{i}} , \widetilde{D_{i}})$ 
is a relative index data over $(\mathbb{R}^{2},  \mathbb{R})$, 
here a continuous coarse map 
$\widetilde{f_{i}} : \widetilde{M_{1}} \to \mathbb{R}^{2}$ 
is defined to be $\widetilde{f_{i}}(x,y) = (x , \mathrm{dist}(y,Z_{i}))$. 
In order to use Lemma \ref{lem:locW}, 
we prove the following. 
This is based on a concept that 
coarse relative index depends only on 
a neighborhood of $W_{i}$; see also 
\cite[Proposition 4.7]{MR3439130}. 

\begin{lem}
\label{lem:cindloc}
By using the map 
$\Gamma : K_{1}(C^{\ast}(W_{1} \subset M_{1})) 
\to K_{1}(C^{\ast}(\widetilde{W_{1}} \subset \widetilde{M_{1}}))$ 
defined by the identity map $N_{r}(W_{1}) \to N_{r}(\widetilde{W_{1}})$, we have 
\[
\Gamma (\text{\rm c-ind}(D_{1} , D_{2})) 
= 
\text{\rm c-ind}(\widetilde{D_{1}} , \widetilde{D_{2}}). 
\]
\end{lem}

\begin{proof}
Take unitaries $U$ and $\widetilde{U}$ 
appeared in the definition of the projection $q$  
such that 
$U = \widetilde{U}$ on $L^{2}(N_{r}(W_{1}) , S_{1})$  
and 
propagation of $U$ and $\widetilde{U}$ is less than $r/8$. 
Take a function 
$f \in U_{1}(C_{0}(\mathbb{R}))$ such that 
$\mathrm{Supp}(\hat{f}) \subset (-r/4 , r/4)$ and 
$[f] = \left[ \frac{x-i}{x+i} \right] \in K_{1}(C_{0}(\mathbb{R}))$. 
We have
\[
\text{\rm c-ind}(D_{1} , D_{2}) = 
\left[ 
f(D_{1})U^{\ast}f(D_{2})^{\ast}U 
\right]
\]
and 
\[
\text{\rm c-ind}(\widetilde{D_{1}} , \widetilde{D_{2}}) = 
\left[ 
f(\widetilde{D_{1}})\widetilde{U}^{\ast}
f(\widetilde{D_{2}})^{\ast}\widetilde{U} 
\right]. 
\]


Since operators $f(D_{1}) - U^{\ast}f(D_{2})U$ 
and 
$f(D_{1})^{\ast} - U^{\ast}f(D_{2})^{\ast}U$
are supported in $N_{r/4}(W_{1})$, 
an operator $f(D_{1})U^{\ast}f(D_{2})^{\ast}U - 1$ 
is supported in $N_{r/4}(W_{1})$. 
Since 
propagation of $f(D_{1})U^{\ast}f(D_{2})^{\ast}U$ 
is less than $3r/4$,  
we have 
\begin{align*}
\Gamma (\text{\rm c-ind}(D_{1} , D_{2})) 
&= 
\left[ 
f(D_{1})U^{\ast}f(D_{2})^{\ast}U 
\right] \\ 
&= 
\left[ 
f(\widetilde{D_{1}})\widetilde{U}^{\ast}
f(\widetilde{D_{2}})^{\ast}\widetilde{U} 
\right] \\
&= 
\text{\rm c-ind}(\widetilde{D_{1}} , \widetilde{D_{2}}). 
\end{align*}
 
\end{proof}

Finally, we comlete the proof of Theorem \ref{thm:main1} 
for the product case. 
By Lemma \ref{lem:locW}, 
we have 
\[
\zeta_{\ast} (\text{\rm c-ind}(D_{1} , D_{2})) = 
\widetilde{\zeta}_{\ast}(\Gamma (\text{\rm c-ind}(D_{1} , D_{2}))). 
\]
By Lemma \ref{lem:cindloc}, 
the value equals 
$\widetilde{\zeta}_{\ast}(\text{\rm c-ind}(\widetilde{D_{1}} , \widetilde{D_{2}}))$. 
Since $\widetilde{N_{1}}$ and $\widetilde{N_{2}}$ is closed 
and we can take a unitary $\widetilde{U}$ satisfies 
$\widetilde{U}\widetilde{\Pi_{1}} = \widetilde{\Pi_{2}}\widetilde{U}$, 
we have 
\begin{align*}
\widetilde{\zeta}_{\ast}(\text{\rm c-ind}(\widetilde{D_{1}} , \widetilde{D_{2}})) 
&= -\frac{1}{8\pi i}
\ind \left( \widetilde{\Pi_{1}} \frac{\widetilde{D_{1}} - i}{\widetilde{D_{1}}+i} 
\widetilde{U}^{\ast}\frac{\widetilde{D_{2}} + i}{\widetilde{D_{2}} - i}
\widetilde{U} \widetilde{\Pi_{1}}\right) \\ 
&= -\frac{1}{8\pi i}
\ind \left( 
\widetilde{\Pi_{1}} 
\frac{\widetilde{D_{1}} - i}{\widetilde{D_{1}}+i} \widetilde{\Pi_{1}} \right) 
+ \frac{1}{8 \pi i}
\ind \left( 
\widetilde{\Pi_{2}} 
\frac{\widetilde{D_{2}} - i}{\widetilde{D_{2}}+i} \widetilde{\Pi_{2}} \right) . 
\end{align*}
By an index theorem on partitioned mannifolds  
\cite[Theorem 3.3]{MR996446},  
the value equals
\[ 
-\frac{1}{8\pi i}
\left( \ind (\widetilde{D_{N_{1}}}^{+}) - \ind (\widetilde{D_{N_{2}}}^{+}) \right). 
\]
This is nothing but the relative topological index, 
so that the proof is completed. 

\section{Relative index formula on even dimension}
\label{sec:even}

In this section, 
we state and prove an index theorem for 
an even relative index data partitioned by 
submanifolds of codimension $1$. 
This index formula is a counterpart of Theorem \ref{thm:main1}. 

\subsection{Index theorem}

Let $(M_{i} , W_{i} , D_{i})$ be an even relative index data 
over $(M,W)$ 
partitioned by $(N_{1} , N_{2})$. 
The Dirac operator $D_{i}$ induces 
a Dirac operator $D_{N_{i}}$ on 
a Clifford bundle $S_{N_{i}} = S^{+}|_{N_{i}} \to N_{i}$ 
and 
they satisfies 
$(\psi_{N_{2} \setminus Z_{2}})^{\ast} \circ D_{N_{1}} 
= D_{N_{2}} \circ (\psi_{N_{2} \setminus Z_{2}})^{\ast}$. 
We denote by $(\phi_{N_{1}}, \phi_{N_{2}})$ 
the pair of restriction of functions 
$(\phi_{1} , \phi_{2}) \in GL_{l}(\mathfrak{W})$ to 
$N_{1}$ and $N_{2}$, respectively. 
Then we have 
$\phi_{N_{1}} \circ \psi  |_{Z_{2}^{c}} = \phi_{N_{2}} |_{Z_{2}^{c}}$. 

We define 
the relative topological Toeplitz index 
$\text{\rm ind}_{t}(\phi_{N_{1}} , D_{N_{1}} , \phi_{N_{2}}, D_{N_{2}})$. 
Let $\widetilde{N_{i}}$ be a closed manifold 
such that 
$\Omega_{i} \subset \widetilde{N_{i}}$
and $\widetilde{S_{N_{i}}} \to \widetilde{N_{i}}$
a Clifford bundle  as in subsection \ref{subsec:odd}. 
Namely, we assume that 
all structures on 
$\Omega_{i} \subset N_{i}$ are isomorphic to those on 
$\Omega_{i} \subset \widetilde{N_{i}}$, respectively, 
and 
there exists an isometry 
$\widetilde{\psi} : \widetilde{N_{2}} \setminus Z_{2} 
	\to \widetilde{N_{1}} \setminus Z_{1}$ 
such that $\widetilde{\psi}$ induces isometry of graded Clifford bundles 
$\widetilde{\psi}^{\ast} : 
\widetilde{S_{N_{1}}}|_{\widetilde{N_{1}} \setminus Z_{1}} 
	\to \widetilde{S_{N_{2}}}|_{\widetilde{N_{2}} \setminus Z_{2}}$. 
There is the Dirac opeator $\widetilde{D_{N_{i}}}$ on 
$\widetilde{S_{N_{i}}}$. 
Take $\widetilde{\phi_{N_{i}}} \in GL_{l}(C(\widetilde{N_{i}}))$ such that 
$\phi_{N_{i}}|_{\Omega_{i}} 
= \widetilde{\phi_{N_{i}}}|_{\Omega_{i}}$ 
and 
$\widetilde{\phi_{N_{1}}} \circ \widetilde{\psi} 
= 
\widetilde{\phi_{N_{2}}}|_{\widetilde{N_{2}} \setminus Z_{2}}$. 
Denote by $\mathcal{H}_{i}$ the subspace of 
$L^{2}(\widetilde{N_{i}}, \widetilde{S_{N_{i}}})$ 
generated by 
the non-negative eigenvectors of $\widetilde{D_{i}}$ 
and let 
$P_{i} : L^{2}(\widetilde{N_{i}}, \widetilde{S_{N_{i}}})^{l} \to \mathcal{H}_{i}^{l}$ 
be the projection. 
Then for any $s \in \mathcal{H}_{i}^{l}$, 
we define the Toeplitz operator 
$T_{\widetilde{\phi_{N_{i}}}} : \mathcal{H}_{i}^{l} \to \mathcal{H}_{i}^{l}$ 
by $T_{\widetilde{\phi_{N_{i}}}}(s) = P_{i}\widetilde{\phi_{N_{i}}}s$.  
The Toeplitz operator $T_{\widetilde{\phi_{N_{i}}}}$ is Fredholm 
since the values of $\widetilde{\phi_{N_{i}}}$ are in $GL_{l}(\mathbb{C})$. 
Set 
\[
\text{\rm ind}_{t}(\phi_{N_{1}} , D_{N_{1}} , \phi_{N_{2}}, D_{N_{2}}) 
= 
\ind \left(T_{\widetilde{\phi_{N_{1}}}} \right) 
- \ind \left(T_{\widetilde{\phi_{N_{2}}}} \right) \in \mathbb{Z}. 
\]
The value is independent of the choice of 
a compactification $\widetilde{N_{i}}$, 
a Clifford bundle $\widetilde{S_{N_{i}}}$ and 
a function $\widetilde{\phi_{N_{i}}}$. 
This is essentially due to 
\cite[Proposition 4.6]{MR720933}.  

Our second main theorem is the following. 
This is a counterpart of Theorem \ref{thm:main1} 
and also 
a generalization of 
\cite[Theorem 2.6]{seto2}.   

\begin{thm}
\label{thm:main2}
Let $M$ be a complete Riemannian manifold and 
$W \subset M$ a closed subset. 
Let $(M_{i}, W_{i} , D_{i})$ be an even
relative index data over $(M,W)$ 
which is partitioned by $(N_{1} , N_{2})$. 
Take $(\phi_{1} , \phi_{2}) \in GL_{l}(\mathfrak{W})$. 
Then the following formula holds: 
\[
\zeta_{\ast}(\text{\rm c-ind}(\phi_{1} , D_{1} , \phi_{2} , D_{2})) 
= -\frac{1}{8\pi i}\text{\rm ind}_{t}(\phi_{N_{1}} , D_{N_{1}} , \phi_{N_{2}} , D_{N_{2}}) . 
\]
\end{thm}

We prove Theorem \ref{thm:main2} in Subsection \ref{subsec:pr2}. 
In the proof of Theorem \ref{thm:main2},  
we do not use the fact that 
the relative topological Toeplitz index 
$\text{\rm ind}_{t}(\phi_{N_{1}} , D_{N_{1}} , \phi_{N_{2}} , D_{N_{2}})$ 
does not depend on the choice of 
a compactification $\widetilde{N_{i}}$, 
a Clifford bundle $\widetilde{S_{N_{i}}}$ 
and 
a  function $\widetilde{\phi_{N_{i}}}$.  
Thus Theorem \ref{thm:main2} gives a new proof 
of well definedness of 
$\text{\rm ind}_{t}(\phi_{N_{1}} , D_{N_{1}} , \phi_{N_{2}} ,  D_{N_{2}})$. 
That is the same as the case of 
$\text{\rm ind}_{t}(D_{N_{1}} , D_{N_{2}})$. 

\subsection{Proof}
\label{subsec:pr2}

In this subsection, 
we prove Theorem \ref{thm:main2}. 
There are $2$ steps to prove it, which is similar to Subsection \ref{subsec:pr1}. 
Namely, the first step is the reduction to the product case and 
the second one is the proof of the product case. 

Similar argument in Subsection \ref{thm:main1} implies 
we can reduce the product case. 
A function $\phi_{i}'$ on $M_{i}'$ 
is taken as $\phi_{1}' \circ \psi' = \phi_{2}'$ 
on $M_{2}' \setminus W_{2}'$, 
$\phi_{i}' = \phi_{i}$ on $M^{+}$ 
and 
$\phi_{i}' = 1 \otimes \phi_{N_{i}}$ on 
$(-\infty , -r] \times N_{i}$. 
The counterpart of Lemma \ref{lem:red1} 
is as follows. 

\begin{lem}
We have 
\[
(\zeta_{r})_{\ast}(\text{\rm c-ind}(\phi_{1} , D_{1} , \phi_{2} , D_{2})) 
= 
\zeta'_{\ast}(\text{\rm c-ind}(\phi_{1}' , D_{1}' , \phi_{2}' , D_{2}')). 
\]
\end{lem}

\begin{proof}
Take unitaries $U$ and $U'$ with propagation less than $r/4$ 
as in the proof of Lemma \ref{lem:red1}. 
Take a chopping function $\chi$ such that 
$\mathrm{Supp}(\hat{\chi}) \subset (-r/8 , r/8)$. 
Then operators 
$u_{\phi_{1}}U^{\ast}u_{\phi_{2}}^{-1}U$ 
and 
$u_{\phi_{1}'}U'{}^{\ast}u_{\phi_{2}'}^{-1}U'$ 
have propagation less than $r$. 
Thus we have 
\[
\Pi_{1,r} u_{\phi_{1}}U^{\ast}u_{\phi_{2}}^{-1}U \Pi_{1,r} 
= 
\Pi_{1}' u_{\phi_{1}'}U'{}^{\ast}u_{\phi_{2}'}^{-1}U' \Pi_{1}'.  
\]
By the way, we recall that 
\[
\text{\rm c-ind}(\phi_{1} , D_{1} , \phi_{2} , D_{2}))
= 
\left[ u_{\phi_{1}}U^{\ast}u_{\phi_{2}}^{-1}U \right]. 
\]
Therefore, we complete the proof 
by using Lemma \ref{lem:Higson}. 
\end{proof}

Threfore, 
we reduced to the product case 
$(M_{i} = \mathbb{R} \times N_{i} , 
W_{i} = \mathbb{R} \times Z_{i} , D_{i})$ 
similar to Subsection \ref{subsec:pr1}. 
Here, 
a function $\phi_{i}$ is given by $\phi_{i} = 1 \otimes \phi_{N_{i}}$. 

Let us prove the product case. 
Take a closed manifold $\widetilde{N_{i}}$ and 
a function $\widetilde{\phi_{N_{i}}}$ 
as in the definition of the relative topological Toeplitz index. 
Set 
$\widetilde{M_{i}} = \mathbb{R} \times \widetilde{N_{i}}$, 
$\widetilde{W_{i}} = \mathbb{R} \times Z_{i} = W_{i}$ 
and 
$\widetilde{\phi_{i}} = 1 \otimes \widetilde{\phi_{N_{i}}}$. 
Then 
$(\widetilde{M_{i}} , \widetilde{W_{i}} , \widetilde{D_{i}})$ 
is an even relative index data over $(\mathbb{R}^{2},  \mathbb{R})$, 
here a continuous coarse map 
$\widetilde{f_{i}} : \widetilde{M_{i}} \to \mathbb{R}^{2}$ 
is defined to be $\widetilde{f_{i}}(x,y) = (x , \mathrm{dist}(y,Z_{i}))$. 
The counterpart of Lemma \ref{lem:cindloc} is as follows. 

\begin{lem}
\label{lem:loccToe}
By using the map 
$\Gamma : K_{1}(C^{\ast}(W_{1} \subset M_{1})) 
\to K_{1}(C^{\ast}(\widetilde{W_{1}} \subset \widetilde{M_{1}}))$ 
defined by the identity map $N_{r}(W_{1}) \to N_{r}(\widetilde{W_{1}})$, we have 
\[
\Gamma (\text{\rm c-ind}(\phi_{1} , D_{1} , \phi_{2} , D_{2})) 
= 
\text{\rm c-ind}(\widetilde{\phi_{1}} , \widetilde{D_{1}} , 
\widetilde{\phi_{2}} , \widetilde{D_{2}}). 
\]
\end{lem}

\begin{proof}
Take unitaries $U$ and $\widetilde{U}$ 
such that 
$U = \widetilde{U}$ on $L^{2}(N_{r}(W_{1}) , S_{1})$ and 
propagation of $U$ and $\widetilde{U}$ is less than $r/8$ 
as in the proof of Lemma \ref{lem:cindloc}. 
Take a chopping function $\chi$ such that 
$\mathrm{Supp}(\hat{\chi}) \subset (-r/8 , r/8)$. 
Then operators  
$u_{\phi_{1}} - U^{\ast}u_{\phi_{2}}U$ 
and 
$u_{\phi_{1}}^{-1} - U^{\ast}u_{\phi_{2}}^{-1}U$ 
are supported in $N_{r/4}(W_{1})$, 
so that  an operator  
$u_{\phi_{1}}U{}^{\ast}
u_{\phi_{2}}^{-1}U - 1$  
is supported in $N_{r/4}(W_{1})$. 
Similarly, 
an operator  
$u_{\widetilde{\phi_{1}}}\widetilde{U}{}^{\ast}
u_{\widetilde{\phi_{2}}}^{-1}\widetilde{U} - 1$  
is supported in $N_{r/4}(\widetilde{W_{1}}) = N_{r/4}(W_{1})$.  
Since propagation of  
$u_{\phi_{1}}U{}^{\ast}u_{\phi_{2}}^{-1}U$ and 
$u_{\widetilde{\phi_{1}}}\widetilde{U}^{\ast}
u_{\widetilde{\phi_{2}}}^{-1}\widetilde{U}$
is less than $3r/4$, 
we have 
\begin{align*}
\Gamma (\text{\rm c-ind}(\phi_{1} , D_{1} , \phi_{2} , D_{2})) 
= 
\text{\rm c-ind}(\widetilde{\phi_{1}} , \widetilde{D_{1}} , 
\widetilde{\phi_{2}} , \widetilde{D_{2}}). 
\end{align*}
\end{proof}

Finally, we comlete the proof of Theorem \ref{thm:main2} 
for the product case. 
By Lemma \ref{lem:locW}, 
we have 
\[
\zeta_{\ast} (\text{\rm c-ind}(\phi_{1} , D_{1} , \phi_{2} , D_{2}) = 
\widetilde{\zeta}_{\ast}(\Gamma (\text{\rm c-ind}(\phi_{1} , D_{1} , \phi_{2} , D_{2}))). 
\]
By Lemma \ref{lem:loccToe}, 
the value equals 
$\widetilde{\zeta}_{\ast}(\text{\rm c-ind}(\widetilde{\phi_{1}} , \widetilde{D_{1}} , 
\widetilde{\phi_{2}} , \widetilde{D_{2}}))$. 
Since $\widetilde{N_{1}}$ and $\widetilde{N_{2}}$ are closed 
and we can take a unitary $\widetilde{U}$ satisfies 
$\widetilde{U}\widetilde{\Pi_{1}} = \widetilde{\Pi_{2}}\widetilde{U}$, 
we have 
\begin{align*}
& \, \widetilde{\zeta}_{\ast}
(\text{\rm c-ind}(\widetilde{\phi_{1}} , \widetilde{D_{1}} , 
\widetilde{\phi_{2}} , \widetilde{D_{2}})) \\ 
=& -\frac{1}{8\pi i}
\ind \left( \widetilde{\Pi_{1}}u_{\widetilde{\phi_{1}}}\widetilde{U}^{\ast}
u_{\widetilde{\phi_{2}}}^{-1}\widetilde{U} \widetilde{\Pi_{1}}\right) \\ 
=& -\frac{1}{8\pi i}
\ind \left( 
\widetilde{\Pi_{1}} 
u_{\widetilde{\phi_{1}}} \widetilde{\Pi_{1}} \right) 
+ \frac{1}{8 \pi i}
\ind \left( 
\widetilde{\Pi_{2}} 
u_{\widetilde{\phi_{2}}} \widetilde{\Pi_{2}} \right) . 
\end{align*}
By an index theorem for Toeplitz operators on partitioned mannifolds  
\cite[Theorem 2.6]{seto2},    
the value equals
\[ 
-\frac{1}{8\pi i}
\left( \ind \left(T_{\widetilde{\phi_{N_{1}}}} \right) 
- \ind \left(T_{\widetilde{\phi_{N_{2}}}} \right)  \right). 
\]
This is nothing but the relative topological Toeplitz index, 
so that the proof is completed.

\bibliographystyle{amsplain}
\bibliography{cyclic_relative_partitioned_ref}

\providecommand{\bysame}{\leavevmode\hbox to3em{\hrulefill}\thinspace}
\providecommand{\MR}{\relax\ifhmode\unskip\space\fi MR }
\providecommand{\MRhref}[2]{%
  \href{http://www.ams.org/mathscinet-getitem?mr=#1}{#2}
}
\providecommand{\href}[2]{#2}
\begin{thebibliography}{10}

\bibitem{MR823176}
A.~Connes, \emph{Noncommutative differential geometry}, Inst. Hautes \'Etudes
  Sci. Publ. Math. (1985), no.~62, 257--360. \MR{823176 (87i:58162)}

\bibitem{MR720933}
M.~Gromov and H.~Blaine Lawson, Jr., \emph{Positive scalar curvature and the
  {D}irac operator on complete {R}iemannian manifolds}, Inst. Hautes \'Etudes
  Sci. Publ. Math. (1983), no.~58, 83--196 (1984). \MR{720933}

\bibitem{MR1113688}
N.~Higson, \emph{A note on the cobordism invariance of the index}, Topology
  \textbf{30} (1991), no.~3, 439--443. \MR{1113688 (92f:58171)}

\bibitem{KSZ-rp}
M.~Karami, A.H.S. Sadegh, and M.E. Zadeh, \emph{Relative-partitioned index
  theorem}, 2014, arXiv:1411.6090.

\bibitem{MR996446}
J.~Roe, \emph{Partitioning noncompact manifolds and the dual {T}oeplitz
  problem}, Operator algebras and applications, {V}ol.\ 1, London Math. Soc.
  Lecture Note Ser., vol. 135, Cambridge Univ. Press, Cambridge, 1988,
  pp.~187--228. \MR{996446}

\bibitem{MR1399087}
\bysame, \emph{Index theory, coarse geometry, and topology of manifolds},
  Providence, RI: American Mathematical Society, 1996. \MR{1399087}

\bibitem{MR3439130}
\bysame, \emph{Positive curvature, partial vanishing theorems and coarse
  indices}, Proc. Edinb. Math. Soc. (2) \textbf{59} (2016), no.~1, 223--233.
  \MR{3439130}

\bibitem{seto2}
T.~Seto, \emph{Toeplitz operators and the {R}oe-{H}igson type index theorem},
  2014, arXiv:1405.4852 (to appear in J. of Noncommut. Geom.).

\bibitem{MR3599501}
\bysame, \emph{Toeplitz operators and the {R}oe-{H}igson type index theorem in
  {R}iemannian surfaces}, Tokyo J. Math. \textbf{39} (2016), no.~2, 423--439.
  \MR{3599501}

\bibitem{MR3153339}
P.~Siegel, \emph{Homological calculations with the analytic structure group},
  ProQuest LLC, Ann Arbor, MI, 2012, Thesis (Ph.D.)--The Pennsylvania State
  University. \MR{3153339}

\end{thebibliography}

\end{document}